\documentclass[a4paper,10pt]{amsart}
\usepackage[utf8x]{inputenc}

\usepackage{amsmath,amsthm,amsfonts,amssymb, mathdots}
\usepackage{epsfig, psfrag}
\usepackage{graphicx,graphics}
\usepackage{latexsym}
\usepackage[all]{xy}
\usepackage{subfigure}
\usepackage[usenames,dvips]{color}  
\usepackage{color} 
\usepackage{psfrag}  

\xyoption{matrix}
\xyoption{arrow}

\theoremstyle{plain}
\newtheorem{lemma}{Lemma}[section]
\newtheorem{prop}[lemma]{Proposition}
\newtheorem{cor}[lemma]{Corollary}
\newtheorem{thm}[lemma]{Theorem}

\newtheorem{de}[lemma]{Definition}

\newtheorem{rem}[lemma]{Remark}
\newtheorem{nota}[lemma]{Notation}

\newcommand{\Z}{\mathbb Z}

\newcommand{\N}{\mathbb N}
\newcommand{\X}{\mathbb X}

\begin{document}
\title{Repetitive higher cluster categories of type $A_n$}

\author[Lamberti]{Lisa Lamberti}
\address{Departement Mathematik\\
ETHZ\\
Rämistrasse 101
8092 Zürich \\
Schweiz
}
\email{lisa.lamberti@math.ethz.ch}


\begin{abstract}
We show that the repetitive higher cluster category of type $A_n$,
defined as the orbit category 
$\mathcal{D}^b(\mathrm{mod}k A_n)/(\tau^{-1}[m])^p$,
is equivalent to a category defined on a subset of diagonals in a regular polygon.
This generalizes the construction of Caldero-Chapoton-Schiffler, \cite{CCS}, which we recover 
when $p=m=1$, and the work of Baur-Marsh, \cite{BM}, treating the case $p=1, m>1$.
Our approach also leads to a geometric model of the bounded derived category 
in type $A$.
\end{abstract}

\maketitle

\section{Introduction}

In this paper we give a geometrical-combinatorial model, in 
the spirit of Caldero-Chapoton-Schiffler, \cite{CCS},
of repetitive higher cluster categories of type $A_n$. These are orbit categories 
of the bounded derived category of $\mathrm{mod}k A_n$ under the action of the cyclic group 
generated by the auto-equivalence 
$(\tau^{-1}[m])^p$, where $kA_n$ is the path algebra associated
to a Dynkin quiver of type $A_n$, $\tau$ is the AR-translation and 
$[m]$ the composition of the shift functor $[1]$ on $\mathcal{D}^b(\mathrm{mod}k A_n)$ with itself $m$-times 
and $1\leq p \in \mathbb{N}$. 
We write $$\mathcal{C}^m_{n,p}:=\mathcal{D}^b(\mathrm{mod}k A_{n})/<(\tau^{-1}[m])^p>.$$
The class of objects of $\mathcal{C}^m_{n,p}$ is the same as for $\mathcal{D}^b(\mathrm{mod}k A_{n})$
and the space of morphisms is as follows
$$
\mathrm{Hom}_{\mathcal{C}^m_{n,p}}(X,Y)=
\bigoplus_{i\in\Z} \mathrm{Hom}_{\mathcal{D}^b(\mathrm{mod}k A_n)}(X, (\tau^{-p}[pm])^i Y).
$$
When the index $p$ or $m$ is equal to one we omit it in the writing of $\mathcal{C}^m_{n,p}$.

When $p=m=1$ one recovers the usual cluster categories, which we denote simply by $\mathcal{C}_n$,
defined independently in \cite{CCS} for the case $A_n$ and in the general case in \cite{BMRRT}.
If $p=1$ and $m>1$ one regains the higher cluster
categories $\mathcal{C}^m_n$ defined in \cite{Keller}, also called $m$-cluster categories.
For $\mathcal{C}^m_n$ Baur-Marsh gave a geometric model in \cite{BM}, 
and in \cite{BM2} for type $D_n$ .
In the case $p>1$ and $m=1$, the category is simply called repetitive cluster category $\mathcal{C}_{n,p}$, 
studied also by Zhu in \cite{Zhu} from a purely algebraic point of view.

The main results of the paper are the following.
On one side we are able to give an equivalence of 
categories between $\mathcal{C}^m_{n,p} $ and a category defined
on a subset of all the diagonals in a regular $p((n+1)m+1)$-gon $\Pi^p$. 
Then, the model we propose here also leads to
a geometric interpretation of cluster tilting objects in $\mathcal{C}^m_{n,p}$.
Furthermore, for $m=1$ we are able to prove an equivalence of triangulated categories between
$\mathcal{C}_{n,p} $ and a quotient of a cluster category of a certain rank.
On the other hand as an application of the results obtained for $\mathcal{C}^m_{n,p}$,
a geometric model for $\mathcal{D}^b(\mathrm{mod}k A_{n-1})$ will be given.

The association of geometric models to algebraic categories has been
studied and developed by many authors, among others we mention: \cite{BM},  
\cite{BM2}, \cite{BZ}, \cite{BuTo}, \cite{CCS}, \cite{Tork},\cite{S}, $\dots$.
This approach is not only beautiful but also fruitful
as it gives new ways to understand the intrinsic combinatorics of the category.

The particularity of the repetitive higher cluster categories is that they
are fractionally Calabi-Yau of dimension $\frac{p(m+1)}{p}$, 
this means that $(\tau[1])^p\cong[p(m+1)]$ as triangle functors,
and the fraction cannot be simplified.
And it is precisely in this point that the category
$\mathcal{C}^m_{n,p}$ differs from $\mathcal{C}_n$ and $\mathcal{C}^m_n$.

This generalization of the notion of a Calabi-Yau category
is interesting because many categories happen to be of this type. Consider
for example the bounded derived category of $\mathrm{mod} kQ$, for a quiver $Q$ with underlying Dynkin diagram, 
or of the category of coherent sheaves on an elliptic curve or a 
weighted projective line of tubular type. 
A.-C. van Roosmalen was recently able to give a classification up to 
derived equivalence of abelian hereditary categories whose bounded
derived category are fractionally Calabi-Yau, see \cite{Va}. 

The structure of the paper is as follows. The next section is dedicated
to a review of useful definitions on fractionally Calabi-Yau categories and 
repetitive cluster categories will be carefully defined. In Section 3 we define 
the repetitive polygon $\Pi^p$ and we choose a subset of diagonals in $\Pi^p$
together with a rotation rule between them. This will lead to the modelling of
$\mathcal{C}_{n,p}$.
In Section 4 we study the relation between repetitive cluster categories
and cluster categories.
Section 5 is dedicated to the study of cluster tilting objects in $\mathcal{C}_{n,p}$.
The content of Section 6 is the geometric modelling of $\mathcal{C}^m_{n,p},$
the construction we give here generalizes the one of Section 3.
Finally, in Section 7 we extend the construction of $\Pi^p$ to a geometric
figure with an infinite number of sides, $\Pi^{\pm \infty}$. 
Applying the results of
Section 7 to the category generated by the set of diagonals complementary to $2$-diagonals in
$\Pi^{\pm \infty}$, we obtain a model for 
$\mathcal{D}^b(\mathrm{mod}k A_{n+1})$.

I would like to thank my advisor Prof. K. Baur for the inspiring discussions we had. 
I am also grateful to Prof. R. Marsh for useful discussions.

\section{Repetitive cluster categories of type $A_n$}

\subsection{Serre duality and Calabi-Yau categories}

Let $k$ be a field (assume it to be algebraically closed, even though in this section it is not needed) 
and let $\mathcal{K}$ be a $k$-linear triangulated category which is $\mathrm{Hom}$-finite,
i.e.\ for any two objects in $\mathcal{T}$ the space of morphisms is a finite dimensional vector space.
\begin{nota}
Throughout the paper we denote by $\mathcal{D}$ the category $\mathcal{D}^b(\mathrm{mod}kA_{n})$ 
or $\mathcal{D}^b(\mathrm{mod} kA_{n-1})$, and we believe that it will
be clear from the context if it is one or the other. 
\end{nota}

\begin{de}
A $k$-triangulated category $\mathcal{K}$ has a {\em Serre functor} if
it is equipped with an auto-equivalence
$\nu:\mathcal{K}\rightarrow\mathcal{K}$ together with bifunctorial isomorphisms
$$
D\mathrm{Hom}_{\mathcal{K}}(X,Y)\cong\mathrm{Hom}_{\mathcal{K}}(Y,\nu X),
$$
for each $X,Y\in\mathcal{K}$. $D$ indicates the vector space duality $\mathrm{Hom}_{k}(?,k)$.
\end{de}
We will say that $\mathcal{K}$ has Serre duality if $\mathcal{K}$ admits a Serre functor.
In the case $\mathcal{K}=\mathcal{D}$ a Serre functor exists (\cite{K2} p. 24), 
it is unique up to isomorphism
and $\nu\stackrel {\sim}{\rightarrow}\tau[1]$, where $\tau$ is the Auslander-Reiten translate and
$[1]$ the shift functor of $\mathcal{D}$.

\begin{de}
A {\em triangle functor} between two triangulated categories $\mathcal{J}$ and $\mathcal{K}$
is a pair $(F,\sigma)$ where $F:\mathcal{J}\rightarrow\mathcal{K}$ is a $k$-linear functor 
and $\sigma:F[1]\rightarrow[1] F$
an isomorphism of functors such that the image of a triangle 
in $\mathcal{J}$ under $F$ is a triangle in $\mathcal{K}$.

Suppose $(F,\sigma)$ and $(G,\gamma)$ are triangle functors, then a {\em morphism of triangle functors}
is a morphism of functors $\alpha:F\rightarrow G$ such that the square
$$
\xymatrix{
F[1] \ar[d]^{\alpha[1]}\ar[r]^{\sigma}  & [1] F \ar[d]^{[1]\alpha} \\
G[1] \ar[r]^{\gamma} & [1] G }
$$
commutes.
\end{de}

\begin{de}
One says that a category $\mathcal{K}$ with Serre functor $\nu$ is a {\em fractionally Calabi-Yau category 
of dimension $\frac{m}{n}$} or {\em $\frac{m}{n}$-Calabi-Yau}
if there is an isomorphism of triangle functors: $$\nu^n\cong[m],$$ 
for $n,m>0$, and where $[m]$ indicates the composition of the shift 
functor with itself $m$ times.
\end{de}

\begin{rem} Notice
that a category of fractional CY dimension $\frac{m}{n}$ is also
of fractional CY dimension $\frac{m t}{nt}$, $t\in\Z$. 
However the converse is not always true.
\end{rem}

\subsection{Repetitive cluster categories of type $A_n$}
In the following we give the algebraic description 
of the repetitive cluster category of type $A_n$. This is the orbit category
of the bounded derived category of $\mathrm{mod} kA_n$ 
under the action of the cyclic group generated
by the auto-equivalence $(\tau^{-1}[1])^p=\tau^{-p}[p]$ for $p>0$, where
$\tau$ is the AR-translation in $\mathcal{D}$ and 
$[1]$ the shift functor. Repetitive cluster categories were defined in \cite{Zhu} as orbit categories of  
$\mathcal{D}^b(\mathcal{H})$, for $\mathcal{H}$ a hereditary abelian category with tilting objects.

\begin{de}\label{defrepcat}
The {\em repetitive cluster category} 
$$\mathcal{C}_{n,p}:=\mathcal{D}/<\tau^{-p}[p]>$$
of type $A_n$,
has as class of objects the same as in $\mathcal{D}$.
The class of morphism is given by:
$$
\mathrm{Hom}_{\mathcal{C}_{n,p}}(X,Y)=
\bigoplus_{i\in\Z} \mathrm{Hom}_{\mathcal{D}}(X, (\tau^{-p}[p])^i Y)
$$
\end{de}
\noindent
Observe that when $p=1$, one gets back the usual cluster category which we simply denote by $\mathcal{C}_n$.
Furthermore, one can define the projection functor $\eta_p:\mathcal{C}_{n,p}\rightarrow \mathcal{C}_n$ which
sends an object $X$ in $\mathcal{C}_{n,p}$ to an object $X$ in $\mathcal{C}_n$, and
$\phi:X\rightarrow Y$ in $\mathcal{C}_{n,p}$ to the morphism $\phi:X\rightarrow Y$ in $\mathcal{C}_{n}$, \cite{Zhu}.
Then one has that $\pi_1=\pi_p\circ\eta_p$, where $\pi_p:\mathcal{D}\rightarrow \mathcal{C}_{n,p}$.

In the following let $\mathcal{F}=\mathcal{F}_1$ be the fundamental domain for the functor
$F:=\tau^{-1}[1]$ in
$\mathcal{D}$
given by the isoclasses of
indecomposables objects in $\mathrm{mod}(kA_n)$ together with the $[1]$-shift of
the projective indecomposable modules. 
After \cite[Proposition: 1.6]{BMRRT} one can identify the subcategory
of isomorphism classes of indecomposable objects of $\mathcal{C}_n$, 
denoted by $\mathrm{ind}(\mathcal{C}_n)$, with
the objects in $\mathcal{F}$.
Let $F^k:=F\circ\dots\circ F$, $k$-times, then denote by
$\mathcal{F}_k$ the $F^k$-shift of $\mathcal{F}$ and we can draw the 
fundamental domain for the functor $\tau^{-p}[p]$ as
in Figure \ref{funddom}.

\begin{figure}[h]
 \unitlength=1cm
   \begin{center}
      \psfragscanon
      \psfrag{1}{$\mathcal{F}$}
      \psfrag{2}{$\mathcal{F}_2$}
\psfrag{4}{$\mathcal{F}_{p-1}$}
\psfrag{3}{$\mathcal{F}_p$}
\psfrag{5}{$\dots$}
\centering
\includegraphics[width=0.7\textwidth]{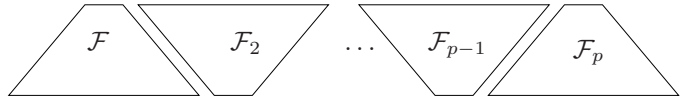} 
   \end{center}
\caption{Partition of the fundamental domain of $\tau^{-p}[p]$ for an odd value of $p$}\label{funddom}
\end{figure} 

As next we summarize some basic properties of $\mathcal{C}_{n,p}$ proven in \cite[Proposition: 3.3]{Zhu}.

\begin{lemma}\label{ZhuProp: 3.3}
\begin{itemize}
 \item $\mathcal{C}_{n,p}$ is a triangulated category with AR-triangles
 and Serre functor $\nu:=\tau[1]$.
 \item The projections $\pi_p:\mathcal{D}\rightarrow \mathcal{C}_{n,p}$
 and $\eta_p:\mathcal{C}_{n,p}\rightarrow \mathcal{C}_n$ are triangle functors. 
 \item $\mathcal{C}_{n,p}$ is fractionally CY of dimension $\frac{2p}{p}$.
 \item $\mathcal{C}_{n,p}$ is a Krull-Schmidt category.
 \item $\mathrm{ind}(\mathcal{C}_{n,p})=\bigcup_{i=1}^{p} \mathrm{ind}(\mathcal{F}_i)$.
\end{itemize}
\end{lemma}


\section{Geometric model of $\mathcal{C}_{n,p}$ }

The geometric model for $\mathcal{C}_{n,p}$ we present in this paper can be viewed as
a $p$-covering of the polygon (which we view as a disc) modelling $\mathcal{C}_{n}$, \cite{CCS}. 
However, the cover we propose here differs
from the usual understanding of a covering space. In fact,the nature of $\mathcal{C}_{n,p}$ suggests
a connection between the first and last layer of the covering. This is the reason why the different ``layers'' 
in our model arise inside a disc as shown in \ref{polrep}.

\subsection{The repetitive polygon $\Pi^p$}\label{sectionreppolygon}
Let $p>1$.
For the purpose of the geometric characterization of $\mathcal{C}_{n,p}$
let $\Pi$ be a regular $N:=n+3$-gon and
let $\Pi^p$ be a regular $p(n+2)$-gon. Number the  
vertices of $\Pi^p$ clockwise repeating
$p$-times the $N-$tuple $1,2,...,N-1, N$ and letting correspond $N\equiv 1.$
Then we denote by $\Pi_1$ a region homotopic to $\Pi$ inside $\Pi^p$ delimited by the segments 
$(1,2),(2,3),\dots,(N-1,N)$ and the inner arc $(1,N)$.
 
Denote by $\rho:\Pi^p\rightarrow\Pi^p$ 
the clockwise rotation through $\frac{2\pi}{p}$ around the center
of $\Pi^p$, and set $\Pi_k:=\rho^{k-1} (\Pi_1)$ for $1\leq k \leq p$. In this way
we divide $\Pi^p$ into $p$ regions. 
See Figure \ref{polrep} where we illustrate this construction. 
\begin{de}\label{defbiggon} 
We call {\em diagonals} of $\Pi^p$ the union of all diagonals of
$$\Pi_1\cup\Pi_2\cup\dots\cup\Pi_{p}.$$
\end{de}
Denote the diagonals of $\Pi^p$ by the triple $(i,j,k)$,
where $1 \leq k\leq p$ specifies a region $\Pi_k$ inside $\Pi^p$, and the tuple
$(i,j)$ defines the diagonal in $\Pi_k$. 

Notice that for us the set of diagonals of $\Pi^p$ consists of a subset of all the straight inner lines
(drawn as arcs)
joining the vertices of $\Pi^p$. Furthermore, the arcs $(1,N,k)$ for $1 \leq k\leq p$ are not diagonals of
$\Pi^p$ as they correspond to boundary segments of $\Pi_k$.
\begin{nota} Observe once and for all that in the writing $(i,j,k)$
we understand that the index $k$ has to be taken modulo $p$, and the indices $i,j$ modulo $N$. 
Furthermore, we always assume that $i<j$.
\end{nota}

\begin{figure}[h]
 \unitlength=1cm
   \begin{center}
      \psfragscanon
      \psfrag{tex1}{$\Pi_1$}
      \psfrag{tex2}{$\Pi_2$}
\psfrag{tex3}{$\Pi_3$}
\centering
\includegraphics[width=0.3 \textwidth]{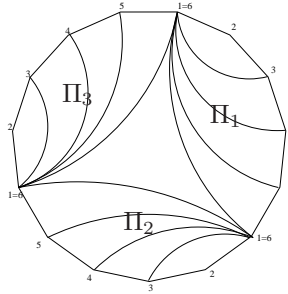} 
   \end{center}
\caption{The polygon $\Pi^3$ for $n=3$}\label{polrep}
\end{figure}

\subsection{Quiver of diagonals of $\Pi^p$} 
As next we associate a translation quiver to the diagonals of $\Pi^p$
with the intention of modelling the AR-quiver of the category $\mathcal{C}_{n,p}.$
The first part of the next definition goes back to \cite{CCS}, the second part is new
and essential for the modelling of the repetitive cluster category.

In the following we call a rotation of a diagonal around a fixed vertex 
{\em irreducible} whenever the other vertex of the diagonal moves 
to the preceding or successive vertex.

\begin{de}\label{derot}
Let $\Gamma_{n,p}$ be the quiver whose vertices are the diagonals $(i,j,k)$
of $\Pi^p$, and whose arrow are defined as follows.
\begin{description}
 \item [1.] For $1\leq i< j<N$: 
$$
{
\xymatrix@-8mm{
 & & (i,j+1,k)\\
\ar[rru]\ar[rrd]  
(i,j,k)\\
 & & (i+1,j,k) 
}}
$$
That is, we draw an arrow if there is an irreducible clockwise rotation around the vertex $j$  or $i$
inside $\Pi_k$, 
\item[2.] For $1\leq i< j=N$ we link
$(i,N,k)\rightarrow (i+1,N,k)$ whenever there is an irreducible  clockwise rotation
around the vertex $N$ inside $\Pi_k$
\item [3.] For $1\leq i< j=N$ we link 
$(i,N,k)\rightarrow (1,i,k+1)$. That is, we draw an arrow to the diagonal
we can reach with a clockwise rotation around vertex $i$ inside $\Pi_k$ 
composed with $\rho$. I.e.\ $(i,N,k)$
is linked to $\rho(1,i,k)=(1,i,k+1)$. 
\end{description}
The set of the operations $1.$ and $2.$ are denoted by {\em $\mathrm{ IrrRot}$}
and the operations $3.$ by {\em $\mathrm{Irr}\rho\mathrm{Rot}$}.
\end{de}
Notice that $\Gamma_{n,p}$ lies on a M\"{o}bius strip when $p$ is odd, and on a cylinder
when the value of $p$ is even.

As next we equip $\Gamma_{n,p}$ with a translation
$
\tau:\Gamma_{n,p}\rightarrow\Gamma_{n,p}
$
such that $(\Gamma_{n,p},\tau)$ becomes a stable translation quiver in the sense of Riedtmann, \cite{Riedt}.

For the reader familiar with the article of \cite{CCS} we point out that the first part of the following 
definition agrees with the one given there, however
the second is new and can be thought of as a connecting map between the different regions $\Pi_k$ in $\Pi^p$.
\begin{de}
The {\em translation $\tau$} on $\Gamma_{n,p}$ maps
\begin{itemize}
 \item $(i,j,k)$ to $(i-1,j-1,k)$ for $ 1<i,j\leq N$ and $1\leq k\leq p$, i.e.\
       $\tau$ is an anticlockwise rotation around the center of $\Pi^p$ by $\frac{2\pi}{p(n+2)}$.
 \item  $(1,j,k)$ to $(j-1,N,k-1)$, i.e. the effect of $\tau$ on $(1,j,k)$ is induced by 
$\frac{1}{N}$th anticlockwise rotation in the regular $N$-gon $\Pi$ homotopic to
$\Pi_{k}$, composed  with $\rho^{-1}$. 
That is $$(1,j,k)\mapsto(j-1,N,k)\mapsto(j-1,N,k-1)=\rho^{-1}(j-1,N,k).$$
\end{itemize}
\end{de}

\begin{lemma}\label{stabletransquiv}
The pair $(\Gamma_{n,p},\tau)$ is a stable translation quiver.
\end{lemma}

\begin{proof}
It is clear that the map $\tau$ is bijective. As $\Gamma_{n,p}$ is finite, 
we only need to persuade us that
the number of arrows from a diagonal $D$ to $D'$ is equal to the number of
arrows from $\tau D'$ to $D$. As there is at most one arrow between any two diagonals, 
we only have to check that
there is an arrow from $D$ to $D'$ if and only if there is an arrow from $\tau D'$ to $D$.

For a given diagonal $D=(i,j,k)$ we distinguish two possible cases. Either there is an arrow to
a diagonal $D'$ of $\Pi_k$ to a diagonal of the neighbouring region $\Pi_{k+1}$.
\begin{itemize}
\item In the first case we distinguish depending on whether $i=1$, or $i\neq 1$.
If $i\neq1$ the result follows from Proposition 2.2 in \cite{BM}.
If $i=1$, then
$$
{\xymatrix@-8mm{
 & & (1,j+1,k)\\
\ar[rru]\ar[rrd]  
(1,j,k)\\
 & & (2,j,k) 
}}
$$ 
Then $\tau(1,j+1,k)=(j,N,k-1),$ and 
$(j,N,k-1)\rightarrow (1,j,k)$. Similarly $\tau(2,j,k)=(1,j-1,k)$ and  $(1,j-1,k)\rightarrow(1,j,k)$.
\item If there is an arrow $(i,j,k)\rightarrow(i',j',k+1)$, we deduce that 
$$j=N,i'=1,\textrm{ and }j'=i\neq 1.$$ Then
$\tau(1,i,k+1)=(i-1,N,k)$, and we deduce that $$(i-1,N,k)\rightarrow(i,N,k).$$
\end{itemize}
\end{proof}

\subsection{Category of diagonals in $\Pi^p$}
We now associate an additive category $\mathcal{C}(\Pi^p)$ to the diagonals of $\Pi^p$. The new
feature arising here is that we only allow a subset of all
possible diagonals of $\Pi^p$, and that we deal with an additional type 
of rotation.

As next we remind the reader the definition of mesh category.
This will be used repeatedly in the next sections.

Let $\alpha$ denote the arrow $x\rightarrow y$, then $\sigma(\alpha)$
denotes the arrow $\tau(y)\rightarrow x$.
\begin{de}
The {\em mesh category} of a translation quiver $(Q,\tau)$ is the
factor category of the path category of $Q$ modulo the ideal generated by the {\em mesh relations}
$$
r_v:=\sum_{\alpha:u\rightarrow v}\alpha\cdot\sigma(\alpha),
$$
where the sum is over all arrows ending in $v$, and $v$ runs through the vertices of $Q$.
\end{de}
Then the category of diagonals  $\mathcal{C}(\Pi^p)$ in $\Pi^p$ 
arises as the mesh category of $(\Gamma_{n,p},\tau)$.
More specifically, the class of objects 
of $\mathcal{C}(\Pi^p)$ is given by formal direct sums of the diagonals in $\Pi^p$, 
i.e.~diagonals in the regions 
$\Pi_k$ for $1\leq k\leq p$. 
The class of morphisms
is generated by the two rotations $\mathrm{IrrRot}_{\Pi^p}$ and 
$\mathrm{Irr}\rho\mathrm{Rot}_{\Pi^p}$, 
carefully defined
in Definition \ref{derot}, modulo the mesh relations which can be read off 
from $(\Gamma_{n,p},\tau)$.

Note that with our approach a new geometric type of mesh relation appears. 
It arises between diagonals of two consecutive regions: $\Pi_k$ and $\Pi_{k+1}$.

\subsection{Equivalence of categories}

The next results will show that the category $\mathcal{C}(\Pi^p)$ 
of diagonals in $\Pi^p$ is equivalent
to the repetitive cluster category of type $A_n$. 
The notation is as specified after Lemma  \ref{ZhuProp: 3.3}.

\begin{prop}\label{bijvarphi}
There is a bijection
$$
\varphi:X\mapsto D_X,
$$
from $\mathrm{ind}(\mathcal{C}_{n,p})$ to the diagonals of $\Pi^p$ such that:
\begin{itemize}
\item $\mathrm{Irr}_{\mathcal{C}_{n,p}}(X,Y)
\leftrightarrow
\mathrm{IrrRot}_{\Pi^p}(D_X,D_Y)$ if $ X, Y $ are in $\mathcal{F}_i$, for some $i$.
\item $\mathrm{ Irr}_{\mathcal{C}_{n,p} } (X,Y)
\leftrightarrow
\mathrm{Irr}\rho\mathrm{Rot}_{\Pi^p}(D_X,D_Y)$ if $X\in\mathcal{F}_i$ and $Y\in\mathcal{F}_j,$ $i\neq j$.
\end{itemize}
\end{prop}
\begin{proof} We saw in Lemma \ref{ZhuProp: 3.3} that 
$$\mathrm{ind}(\mathcal{C}_{n,p})=\bigcup_{i=1}^{p} \mathrm{ind}(\mathcal{F}_i).$$
Then one easily sees that $\varphi$ is a bijection on the level of objects. In fact, we can apply
\cite[Corollary 4.7]{CCS}  to every
subcategory $\mathrm{ind}(\mathcal{F}_i)$ of $\mathcal{C}_{n,p}$
and to the corresponding region $\Pi_{i}$ in $\Pi^p$, $1\leq i \leq p.$

Then for each $1\leq i\leq p$, the irreducible morphisms
between objects $X,Y$ in $\mathcal{F}_i$ where $X\not\cong (\tau^{-1}[1])^{i-1}P[1]$ for $P$
some projective, or where 
$
X\cong (\tau^{-1}[1])^{i-1}P[1]$ and $Y\cong (\tau^{-1}[1])^{i-1}Q[1], 
$
for projectives $P$ and $Q$, agree by 
\cite[Theorem 5.1]{CCS} with the rotations $\mathrm{ IrrRot}_{\Pi_{i} } (D_X,D_Y)$ 
described in part 1. and 2. of Definition \ref{derot}.

Therefore, we only need to study the 
correspondence between $\mathrm{ Irr}_{\mathcal{C}_{n,p} } (X,Y)$ 
and the rotations in $\Pi^p$ described in 
part $3.$ of Definition \ref{derot} for 
$X\cong ((\tau^{-1}[1])^{i-1})\widetilde{P}[1]$ and 
$Y\not\cong (\tau^{-1}[1])^{i-1}\widetilde{Q}[1], 1\leq i\leq p.$
Thus, assume $\mathrm{ Irr}_{\mathcal{C}_{n,p} } (X,Y)\neq 0,$
in this case it follows from the shape of
the AR-quiver of  $\mathcal{C}_{n,p}$ 
that $Y\cong(\tau^{-1}[1])^{i}R,$ for some projective $R$.
Under the bijection $\varphi$ applied to the objects $X$ and $Y$
one then has:
$$ X  \mapsto (i,N,k),\hspace{0,5cm} Y   \mapsto (1,i,k+1).$$
From
Definition \ref{derot} it then follows 
that there is an arrow $(i,N,k)\rightarrow(1,i,k+1)$ in $\Gamma_{n,p}$
defining a corresponding operation in $\mathrm{Irr}\rho\mathrm{Rot}_{\Pi^p}(D_X,D_Y)$ .
Similarly one shows the other direction.
This proves that the mapping $\varphi$ is a bijection also on the level of
morphisms.  

It only remains to check
that the mesh relations in the two categories agree. But by the above we have precisely showed that 
the AR-quiver of $\mathcal{C}_{n,p}$ and 
$\Gamma_{n,p}$ are isomorphic, hence the mesh relations in both categories
agree.
\end{proof}

It follows from the previous result that the projection functor
$\eta_p:\mathcal{C}_{n,p}\twoheadrightarrow\mathcal{C}_n$ corresponds to the
projection $\mu_p:\Pi^p\rightarrow \Pi$.

\section{$\mathcal{C}_{n,p}$ and the link to the cluster category}

The desire of comparing the category $\mathcal{C}_{n,p}$ with the cluster
category of type $A_t$ 
for a certain $t$ arose while 
studying the geometrical model of $\mathcal{C}_{n,p}$. 
We will see in this section that the connection is slightly different than expected, as $t$ cannot simply be
deduced from the size of $\Pi^p$. This is because 
the class of morphisms in $\mathcal{C}_{n,p}$ is too particular. 

However, starting with
a bigger polygon modelling the cluster category, the category $\mathcal{C}_{n,p}$ appears whenever
we assume that $p\neq2.$ In fact, otherwise 
the diagonals of the two models overlap hence the description degenerates.

In the following denote by $(\Gamma_{n,p},\tau_{n,p})$ 
the AR-quiver of $\mathcal{C}_{n,p}$ 
and by $(\Gamma_t,\tau_t)$ the AR-quiver of $\mathcal{C}_t$.

\begin{lemma}\label{lemmasub} Let $p>2.$
Then $(\Gamma_{n,p},\tau_{n,p})$ is a subquiver of $(\Gamma_t,\tau_t)$ for a suitable value of $t$, namely
\begin{itemize} 
	\item  $t:=(n+3)(\frac{p}{2})-3$, if $p$ is even,
	\item  $t:=(n+3)p-3$, if $p$ is odd.
\end{itemize}
\end{lemma}

\begin{proof} 
To prove the claim we establish an isomorphism
of stable translation quivers between $(\Gamma_{n,p},\tau_{n,p})$
and a  stable translation subquiver of $(\Gamma_t,\tau_t)$ for the different 
cases. 

Let us first study the case where $p$ is even.
We observe that the union of the $\tau_t$-orbits of the top and bottom $n$ rows
of the quiver $\Gamma_t$, illustrated in the gray strip in Figure \ref{fig.rep},
defines a subquiver $\tilde{\Gamma}_t$ of $\Gamma_t$. Since $p> 2$,
the top and the bottom $n$ rows do not overlap.

As $\Gamma_t$ is a stable
translation quiver, the same remains true for $\tilde{\Gamma}_t$.
And it turns out that when $t$ is as given in the claim,
$\tilde{\Gamma}_t$ and $\Gamma_{n,p}$  are isomorphic 
as stable translation quivers. First: by construction, the two quivers
have the same number of rows (namely $n$). To see the
isomorphism just on the level of quivers we compare the 
induced action of the auto equivalence $(\tau_{n,p}^{-1}[1])^p$ 
on $\Gamma_{n,p}$ with the action of $(\tau_{t}^{-1}[1])|_{\tilde{\Gamma}_t}$ on 
$\tilde{\Gamma}_t$: It is easy to check that the actions coincide.
Furthermore, because the meshes of the quivers $\tilde{\Gamma}_t$  and  $\Gamma_{n,p}$ 
coincide we deduce
that this gives an isomorphism of stable translation quivers.

When $p$ is odd, the embedding of $\Gamma_{n,p}$ into $\Gamma_t$ can be shown in a 
similar way. However, in this case  
we need to embed $\Gamma_{n,p}$ into the central band of $\Gamma_t$
to preserve the induced action of the autoequivalence, i.e.\ to preserve the
identifications of vertices in $\Gamma_{n,p}$. More precisely,
this time we have to consider 
the horizontal band $n$ vertices wide at the high of $\frac{(p-1)(n+3)}{2}+1$
vertices, counting from the bottom of $\Gamma_t$, see  Figure \ref{fig.rep}. 
Then we concludes as in the previous case.
\end{proof}

\subsection{Triangulated equivalence for $\mathcal{C}_{n,p}$ }
Since the inclusion of the AR-quiver of $\mathcal{C}_{n,p}$ in the AR-quiver of the cluster category
$\mathcal{C}_{t}$ for $t$ as in Lemma \ref{lemmasub} does not give rise to an inclusion at the level of 
full subcategories, $\mathcal{C}_{n,p}$ is in particular not a triangulated subcategory of
$\mathcal{C}_{t}$. However, it is possible to prove that it is
triangulated equivalent to a quotient category of $\mathcal{C}_{t}$.

Quotient categories of the type we are going to consider here
have been studied in \cite{Jorgensen}, et al.~. In particular,
J\o{}rgensen  showed that higher cluster categories are triangulated equivalent
to quotients of the cluster category. We will now recall
the definition of quotient categories and key properties,
which will be needed in the following.

Let $\mathcal{C}$ be an additive category and $\X$ a class of objects of
$\mathcal{C}$. Then the quotient category $\mathcal{C}_{\X}$ has by definition
the same objects as $\mathcal{C}$, but the morphism spaces are taken modulo all 
the morphisms factoring through an object of $\X$. 

Observe that if $\mathcal{C}$ is a triangulated
category, then $\mathcal{C}_{\X}$ needs not to be 
triangulated for all choices of $\X$. However, it
was shown  in \cite[Theorem 2.2]{Jorgensen}
that $\mathcal{C}_{\X}$ is always
pre-triangulated. This means 
that $\mathcal{C}_{\X}$ is 
equipped with a pair $(\sigma,\omega)$ of adjoint endofunctors 
satisfying a number of properties. However none of them
is necessarily an auto-equivalence. Nevertheless, taking a particular
choice of the class $\X$, J\o{}rgensen proved in \cite[ Theorem 3.3 ]{Jorgensen}
that the endofunctors $\sigma$ and $\omega$, 
can be turned into auto-equivalences, so that $(\mathcal{C}_{\X},\sigma)$ becomes a triangulated category. 
We use these facts to link the triangulated structure of $\mathcal{C}_{n,p}$ to a quotient category
of the cluster category of type $A_t$.

\begin{figure}[htbp]
 \unitlength=1cm
   \begin{center}
      \psfragscanon
      \psfrag{te1}{$1$}
      \psfrag{te2}{$n$}
\psfrag{te3}{$t$}
\psfrag{st1}{$1$}
\psfrag{st2}{$a$}
\psfrag{st3}{$a+n-1$}
\psfrag{st4}{$t$}
\centering
\includegraphics[width=0.2 \textwidth]{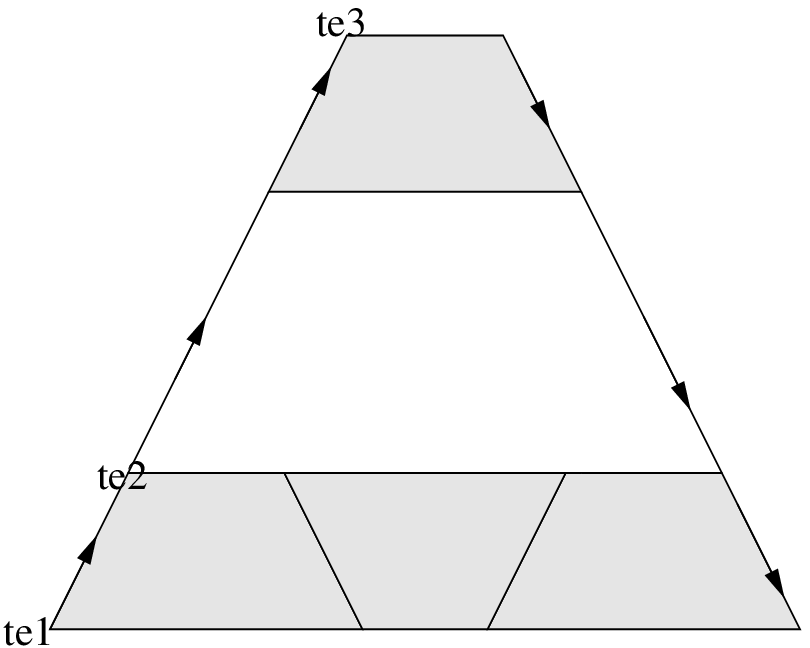}%
\qquad\qquad
\includegraphics[width=0.2 \textwidth]{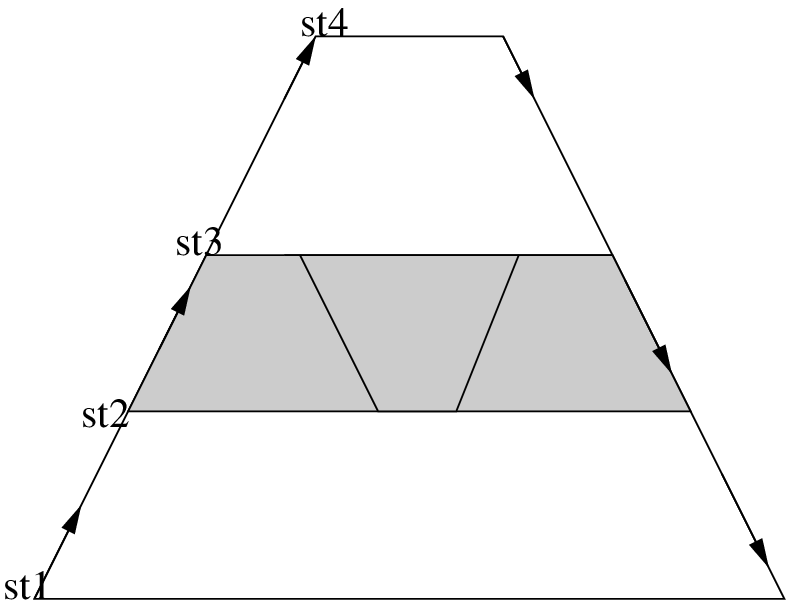}
\caption{Inclusions $\Gamma_{n,4} \subset\Gamma_t$ and 
 $\Gamma_{n,3}\subset\Gamma_t$, and $a=\frac{(p-1)(n+3)}{2}+1$.}
\label{fig.rep}
   \end{center}
\end{figure}

\begin{prop}\label{quot} 
Assume $p\neq2$ and define $t$ as in
Lemma \ref{lemmasub}. Then the category $\mathcal{C}_{n,p}$
is triangulated equivalent 
to a quotient of a cluster category of type $A_t$.
\end{prop}
\begin{proof} 
For $p$ even, denote by $\X$ the additive full subcategory generated by the indecomposable objects
in the band (union of $\tau$-orbits) of $t-2n$ vertices in the middle of $\Gamma_t$, for $t=(n+3)\frac{p}{2}-3$. 
For $p$ odd i.e.~ $t=(n+3)p-3$ one takes as $\X$ the additive full subcategory generated by 
the indecomposable objects
in the band of
$\frac{t-n}{2}$ 
vertices at the top and bottom of $\Gamma_t$.
Then, in both cases we have that $\tau_t \X=\X$. So the AR-quiver of
the quotient category $(\mathcal{C}_t)_{\X}$ is
obtained
by deleting the vertices
corresponding to the objects of $\X$ and the arrows
linked with them (\cite[Theorem 4.2]{Jorgensen}). Furthermore, 
$(\mathcal{C}_t)_{\X}$ is connected, and has finitely many
indecomposable objects up to isomorphism. Furthermore, again by 
\cite[Theorem 4.2]{Jorgensen} $(\mathcal{C}_t)_{\X}$ is standard and of algebraic origin.
Proceeding as in \cite[Theorem 5.2]{Jorgensen} we
conclude that $(\mathcal{C}_t)_{\X}$
is triangulated equivalent to a quotient of a cluster
category of type $A_n$. 

It remains to see that $(\mathcal{C}_t)_{\X}$ and $\mathcal{C}_{n,p}$ 
are equivalent as triangulated categories.
For this we observe that $\mathcal{C}_{n,p}$ is of algebraic origin
by results of \cite[Section 9.3]{K2} and standard by \cite[Proposition 6.1.1.]{C.Amiot}.
Furthermore, it is straightforward to see that
the AR-quivers of $(\mathcal{C}_t)_{\X}$ and
$\mathcal{C}_{n,p}$ are isomorphic as translation quivers.
Thus we are in the conditions of Amiot's Theorem \cite[Theorem 5.1]{Jorgensen}
applied to $(\mathcal{C}_t)_{\X}$ and
$\mathcal{C}_{n,p}$. Hence we deduce that these categories are equivalent as triangulated categories, and so the claim follows.
\end{proof}

\section{Cluster tilting theory for $\mathcal{C}_{n,p}$}
It is well known that cluster tilting objects in cluster categories are strongly 
linked to clusters in cluster algebras discovered by Fomin-Zelevinski in \cite{FZI}. 
This link has been established in \cite{BMRRT}.
Studying the endomorphism algebra of these objects provides a way to 
recover the exchange matrix, indispensable in the mutation process in a cluster algebra.

In this section we are interested in understanding the cluster tilting objects of 
$\mathcal{C}_{n,p}$,
and compare them with configurations of diagonals in the polygon $\Pi^p$. When $p=1$,
it is known that cluster tilting objects of $\mathcal{C}_{n}$ correspond to a 
{\em triangulation} of a regular $n+3$-gon, that is a maximal collection of pairwise non crossing
diagonals.

Cluster tilting objects 
in $\mathcal{C}_{n,p}$
have also been studied from an algebraic point of view in \cite{Zhu}. In fact, Zhu proves that
the endomorphism algebra of a cluster tilting object of $\mathcal{C}_{n,p}$
gives a cover of the endomorphism algebra of the same object, seen now
as a cluster tilting object of the cluster category.

In the following definition let $\mathrm{add}(T)$ be
the full subcategory consisting of
direct summands of direct sums of finitely many copies of $T$.
\begin{de} \label{deftilting}
An object $T \in \mathcal{C}_{n,p}$ is called a {\em cluster tilting object} 
if for any object $X\in\mathcal{C}_{n,p}$  we have that 
$\mathrm{Ext}_{\mathcal{C}_{n,p}}^1(T,X)=0$
if and only if $X\in\mathrm{add}(T)$, and $\mathrm{Ext}_{\mathcal{C}_{n,p}}^1(X,T)=0$
if and only if $X\in\mathrm{add}(T)$.
\end{de}
A cluster tilting object is called basic if all its direct summands are pairwise non isomorphic. 
In this paper all the cluster tilting objects we consider will be basic, so we omit the term basic.
\noindent
The next result follows from \cite[Section 5]{CCS}.
\begin{lemma}\label{tiltinC} 
$T$ is a cluster tilting objects in $\mathcal{C}_n$ if and only if
$\mathcal{X}_T$ is a triangulation of a regular $n+3$-gon.
The cardinality of $\mathcal{X}_T$ is $n$.
\end{lemma}

As our goal is to link cluster tilting objects
of $\mathcal{C}_{n,p}$ to diagonals in $\Pi^p$ we first need to express
what it means geometrically that two objects in $\mathcal{C}_{n,p}$ have no extension.
We will see that it is not enough to study only diagonals in a single region 
$\Pi_k$ in $\Pi^p$, but also diagonals in the regions $\Pi_{k-1}$ respectively
$\Pi_{k+1}$,
depending on which entry of the bifunctor $\mathrm{Ext}^1_{\mathcal{C}_{n,p}}(-,-)$
is fixed.

We remind the reader
that diagonals of $\Pi^p$ consist only of a subset of all the lines joining vertices of $\Pi^p$. 
In the next result we assume that
$i<j$ and $i'<j'$ in the writing of the diagonals $D_X:= (i,j,l)$ and $D_Y:= (i',j',l')$ of $\Pi^p$.
\begin{lemma}\label{exts} Let $D_X:= (i,j,l)$ be fixed and $D_Y:= (i',j',l')$.  
Then 
$$
\mathrm{dim}(\mathrm{Ext}^1_{\mathcal{C}_{n,p}}(X,Y))=1\Leftrightarrow 
\begin{cases}
l=l'  \text{ and }   1\leq i'<i<j'<j\leq N,\\
l=l'+1,  \text{ and } 1\leq i<i'<j<j' \leq N
\end{cases}
$$
Dually, we have
$$\mathrm{dim}(\mathrm{Ext}^1_{\mathcal{C}_{n,p}}(Y,X))=1\Leftrightarrow
\begin{cases}
l'=l, \textrm{ and } 1\leq i< i'<j<j' \leq N,\\
l'=l+1, \textrm{ and} 1\leq i'<i<j'<j\leq N.
\end{cases}$$ 
Otherwise $\mathrm{dim}(\mathrm{Ext}^1_{\mathcal{C}_{n,p}}(-,-))=0$.
\end{lemma}
\begin{proof} When $p=1$, it was remarked in \cite[Section 5]{CCS} that 
$$\mathrm{dim}(\mathrm{Ext}^1_{\mathcal{C}_{n}}(X,Y))=1$$ 
if and only if the corresponding diagonals cross.
For all the other values
of $p$ we have to consider also diagonals in neighbouring regions. 
To compute the $\mathrm{Ext}$-spaces 
we proceed using the fact that
the triangulated category $\mathcal{C}_{n,p}$ has
Serre duality by Lemma \ref{ZhuProp: 3.3}. Thus, in the first case we have
\begin{align*}
D\mathrm{Ext}^1_{\mathcal{C}_{n,p}}(X,-)&=D\mathrm{Hom}_{\mathcal{C}_{n,p}}(X,-[1])\\
&\cong \mathrm{Hom}_{\mathcal{C}_{n,p}}(-[1],\tau X[1])\\
&\cong \mathrm{Hom}_{\mathcal{C}_{n,p}}(-,\tau X).
\end{align*}
Here we used that $\mathcal{C}_{n,p}$ has finite dimensional $\mathrm{Hom}$ spaces, and therefore also finite
dimensional
$\mathrm{Ext}$ spaces.
As next we compute the $\mathrm{Hom}(-,\tau X)$ support using the mesh relations 
in the AR-quiver of $\mathcal{C}_{n,p}$. That gives a rectangular region, also
known as the backwards hammock.
Then one concludes using the bijection between indecomposable objects and diagonals
from Proposition \ref{bijvarphi}.

In the second case, we proceed dually. Then the study the forward hammocks leads to the desired claim.
\end{proof}
With this result in mind we are able to deduce the desired geometric description
of cluster tilting objects, of $\mathcal{C}_{n,p}$, and in fact
this gives a geometric/combinatorial analogue to \cite[Thm: 3.5]{Zhu}
for the case  $\mathcal{H}=\mathrm{mod }k A_n.$
\noindent
Recall that 
the map $\rho:\Pi^p\rightarrow \Pi^p$ is a clockwise rotation
of $\frac{2\pi}{p}$ around the center of $\Pi^p$.
\begin{prop}\label{tiltinbiggon}
$T$ is a cluster tilting object of $\mathcal{C}_{n,p}$ if and only if
$$
\mathcal{X}_T=\mathcal{X}_{T'}\cup\rho(\mathcal{X}_{T'})\cup\dots\cup\rho^{p-1}(\mathcal{X}_{T'})
$$
where $\mathcal{X}_{T'}$ is a triangulation of the region $\Pi_1$.
\end{prop}
\begin{proof}
We saw in Lemma \ref{exts}  
that the non-vanishing of the $\mathrm{Ext}$ spaces
can arise from crossing within the same  $\Pi_k$  or from crossings
between diagonals in the neighbouring regions $\Pi_{k+1}$ and $\Pi_{k-1}$ of $\Pi^p$. 
In particular, diagonals $D_X$ in $\Pi_k$ and
$D_Y$ in $\Pi_{k+1}$ give rise to a non-vanishing
$\mathrm{Ext}$ spaces if and only if $\rho^{-1}(D_Y)$ crosses
$D_X$ inside $\Pi_{k}$  or $\rho(D_X)$ crosses
$D_Y$ inside $\Pi_{k+1}$, and the vertices of the diagonals 
satisfy the conditions of Lemma \ref{exts}. 
Thus,
the only  possible way to get a maximal configuration of diagonals which do not cross
in the sense of both Lemma \ref{exts} is to take the same
triangulation in each region $\Pi_k,$ for $1\leq k\leq p$. 

On the other hand, by Lemma \ref{tiltinC} we have that a triangulation $\mathcal{X}_{T'}$ of a region $\Pi_k,$ corresponds
to a cluster tilting object in $\mathcal{C}_{n}$, then $\rho^i(\mathcal{X}_{T'})$, for
$0\leq i \leq p-1$, gives a 
cluster tilting object $\mathcal{C}_{n,p}$, by the previous argument.
\end{proof}
\begin{cor}
Any cluster tilting object in $\mathcal{C}_{n,p}$ contains
$pn$ pairwise non isomorphic summands.
\end{cor}
\begin{proof}
Since each region $\Pi_k$ is homotopic to a regular $n+3$-gon,  
the set $\mathcal{X}_T$ consists of 
is $pn$ diagonals of $\Pi^p$.
\end{proof}

Observe that in $\mathcal{C}_{n,p}$ there is exactly one complement to an object
in a cluster tilting object $T$, whereas there are exactly two in 
$\mathcal{C}_{n}$. This observation follows from the fact that
whenever a diagonal $D_X$ is removed from $\mathcal{X}_{T}$ as given in Proposition \ref{tiltinbiggon},
and we replace $D_X$ with a unique other
diagonal $D_{X'}$ completing again the triangulation, 
we obtain a new triangulation which is no longer
stable under the rotation of $\rho$.
Thus, to get a new cluster tilting object in $\mathcal{C}_{n,p}$ 
one needs to replace $p$ summands of $T$ corresponding to a complete
$\rho$-orbit of diagonals in $\Pi^p$.
  
\section{Repetitive higher cluster categories $\mathcal{C}^m_{n,p}$ }
In this section we study orbit categories of the form
$$ \mathcal{C}^m_{n,p}:=\mathcal{D}/\langle (\tau^{-1}[m])^p\rangle. $$
We call them {\em repetitive higher cluster categories of type $A_{n}$}.
The class of objects is the same as the class of objects
in $\mathcal{D}$ and the space of morphisms
is given by
$$
\mathrm{Hom}_{\mathcal{C}^m_{n,p}}(X,Y)=\bigoplus_{i\in\Z}\mathrm{Hom}_{\mathcal{D}}(X,(\tau^{-p}[pm])^iY).
$$

We will see below that these categories are interesting because their geometric models
provides us with a geometric description of the bounded derived category
 $\mathcal{D}$.

Repetitive higher cluster categories have similar properties to
the repetitive cluster categories. In particular, identifying
$\mathrm{ind}\mathcal{C}^m$ with $\tilde{\mathcal{F}}=\tilde{\mathcal{F}}_1$, the fundamental
domain for the functor $\tau^{-1}[m]$ in $\mathcal{D}$, and $\tilde{\mathcal{F}}_i$
with the $(\tau^{-1}[m])^{i-1}$-shifts of $\tilde{\mathcal{F}}$ for $1\leq i\leq p,$
we can prove:

\begin{lemma} For all $m\in\N.$
\begin{enumerate}
	\item The projection functors $\pi_p^m:\mathcal{D}\rightarrow\mathcal{C}^m_{n,p}$ and 
        $\eta_p^m: \mathcal{C}^m_{n,p}\rightarrow \mathcal{C}^m_{n}$
	are triangle functors.
        \item The category $ \mathcal{C}^m_{n,p}$ is triangulated, with Serre functor $\tau[1]$ induced from $\mathcal{D}$.
	\item $ \mathcal{C}^m_{n,p}$ is fractionally Calabi Yau of dimension $\frac{p(m+1)}{p}$ and Krull-Schmidt.
	\item $\mathrm{ind}( \mathcal{C}^m_{n,p})=\bigcup_{i=1}^{p}(\tilde{\mathcal{F}}_i)$.
\end{enumerate}
\end{lemma}
\begin{proof}
The first two claims follow from \cite[Theorem 1]{Keller} and \cite[Proposition 1.3]{BMRRT}.
Concerning the fractionally Calabi-Yau dimension, one computes
easily that $(\tau[1])^p\cong[p(m+1)]$ in $ \mathcal{C}^m_{n,p}$.
The last claim can be deduced from a similar reasoning as in \cite[Proposition 1.6]{BMRRT}.
\end{proof}

\subsection{ $\mathcal{C}^m_{n,p}$ via $m$-diagonals in $\Pi^p$}\label{subsecmdiaggeo}
In the following we present a geometric model for the category
$\mathcal{C}^m_{n,p}$ which arises as a generalization
of the one given for $\mathcal{C}_{n,p}$. 
Indecomposable objects are modelled 
by a subset of diagonals in $\Pi^p$
and irreducible morphisms will be associated to rotations of such modulo
the mesh relations.

Let $\Pi^p$ be a regular $p((n+1)m+1)$-gon, $m,n\in\N$.
Then we divide $\Pi^p$ into regions as we did in Section \ref{sectionreppolygon}
when $m=1$. Then each region $\Pi_k$ in $\Pi^p$ 
is homotopic to a regular $((n+1)m+2)$-gon. See also Figure \ref{polrep}.

For the convenience of the reader we recall 
that an {\em $m$-diagonal in a regular $((n+1)m+2)$-gon}, is a diagonal which
divides  the polygon into two parts, each having  $2$ vertices $\mathrm{mod }m$. 
\noindent
Observe that $m$-diagonals have been studied by
many authors, see \cite{BM}, \cite{BM2}, \cite{BuTo}, \cite{FoRe}, 
\cite{ps}, \cite{Thomas}, \cite{Tork}, \cite{Tzanaki}, $\dots$.
We will now adapt the notion of $m$-diagonals 
to the case at hand.

\begin{de} The {\em $m$-diagonals of $\Pi^p$} are 
given by the union of the $m$-diagonals in each 
region $\Pi_k$, for $1\leq k\leq p$.

\end{de}
\begin{nota} In the writing of a diagonal $(i,j,k)$ in $\Pi^p$, we consider the operations
on the first two indices modulo $m$, and the last one modulo $p$.
Remember that the last index specifies the region  $\Pi_k$ in $\Pi^p$. As every where else in the 
this paper we assume that $i<j.$
\end{nota}
We associate a quiver to the $m$-diagonals in $\Pi^p$ as follows.
\begin{de}\label{defmreparrows}
Let $\Gamma^m_{n,p}$ be the {\em quiver of $m$-diagonals of $\Pi^p$}.
Its vertices are the $m$-diagonals, and the arrows are given as follows. 
\begin{description}
\item[1.] If $j\neq N-(m-1)$, 
$$
{\xymatrix@-8mm{
 & & (i,j+m,k)\\
\ar[rru]\ar[rrd]  
(i,j,k)\\
 & & (i-m,j,k)
}}
$$
\item[2.]  If $j= N-(m-1)$ then $(i,N-(m-1),k)\rightarrow (i-m,N-(m-1),k)$.
\item[3.] Furthermore, $(i,N-(m-1),k)\rightarrow (1,i,k+1)$.
\end{description} 
\end{de}
Observe that the  description 1. and 2. 
corresponds to an irreducible $m$-rotations in $\Pi^p$ as defined
in \cite{BM}. That is an
irreducible clockwise rotations between $m$-diagonals around the vertex
$i$ or $j$ inside a region $\Pi_k$. We denote the set of these operations 
by $\mathrm{IrrRot}_m$.
The third type describes the composition of an $m$-irreducible
rotation inside $\Pi_k$ around the vertex
$i$ with the clockwise rotation $\rho(1,i,k)=(1,i,k+1)$. We denote the 
second type of operation by $\mathrm{Irr}\rho\mathrm{Rot}_m$.

Then we equip $\Gamma^m_{n,p}$ with a translation map.
\noindent
The first part of the definition describes an anticlockwise
rotation around the center of $\Pi_k$ through $\frac{2m\pi}{(n+1)m+2}$, the second
defines a composition of a rotation as in the first case, together with
with an anticlockwise rotation around the center of $\Pi^p$ through $\frac{2\pi}{p}$
\begin{de} \label{tau m}
The {\em translation $\tau_m$} maps
\begin{itemize}
\item $(i,j,k)$ to $(i-m,j-m,k)$ if $i,j\neq 1$.
\item If $i=1$,  
$\tau_m(1,j,k)=\rho^{-1}(1-m,j-m,k)=(1-m,j-m,k-1).$ 
\end{itemize}
\end{de}
The proof of the next result is straightforward. We use the fact
that the AR-quiver of $\mathcal{C}^m_{n}$
is isomorphic to the quiver of $m$-diagonals of $\Pi$,  as shown in 
\cite[Proposition 5.4]{BM}, combined with similar arguments as in the proof of 
Lemma \ref{stabletransquiv}

\begin{lemma}\label{transquivermdiag}
There is an isomorphism of stable translation quivers between the
AR-quiver of $\mathcal{C}^m_{n,p}$ and $\Gamma^m_{n,p}$.
\end{lemma}

\subsection{Tilting theory for $\mathcal{C}^m_{n,p}$}
In this section we determine the $m$-cluster tilting
objects of $\mathcal{C}^m_{n,p}$  geometrically, we will see that
they correspond to $\rho$-orbits of $m$-angulations of a region
$\Pi_k$ in $\Pi_p$.
\begin{de} 
An object $T\in\mathcal\mathcal{C}^m_{n,p}$ is called {\em $m$-rigid}
if it is the direct sum of non isomorphic indecomposable objects 
$T_1,\dots, T_t$ such that $\mathrm{Ext}^i_{\mathcal{C}^m_{n,p}}(T_j,T_k)=0$ for all 
$1\leq i \leq m,$ and $1\leq j,k\leq t$. It is called {\em maximal  $m$-rigid} if
it is $m$-rigid and maximal with respect to this property. 

An {\em $m$-cluster tilting object} is an object $T$ in $\mathcal{C}^m_{n,p}$ which is $m$-rigid
and such that $X\in\mathrm{add} T$ if and only if $\mathrm{Ext}^i_{\mathcal{C}^m_{n,p}}(T,X)=0$
for $1\leq i \leq m$, and $X\in\mathrm{add} T$ if and only if $\mathrm{Ext}^i_{\mathcal{C}^m_{n,p}}(X,T)=0$
for $1\leq i \leq m$.
\end{de}
Observe that it was proven in \cite{Wraa} that in  $\mathcal{C}^m_{n}$,
$m$-cluster tilting and maximal $m$-rigid objects coincide.

\begin{de}
An {\em $m$-angulation  of $\Pi^p$} consists
of a maximal set of pairwise non crossing $m$-diagonals in $\Pi^p$. We denote such a set by 
$\mathcal{X}^m$.
\end{de}

For higher cluster categories the following result is a consequence
of \cite[Theorem 1]{Thomas} combined with \cite[Section 4.]{BM}.
\begin{lemma}\label{mang}
$m$-cluster tilting objects in $\mathcal{C}^m_{n} $ correspond to
$m$-angulations $\mathcal{X}^m$ of a regular $(n+1)m+2$-gon.
\end{lemma}
Thus, an $m$-cluster tilting object in $\mathcal{C}^m_{n}$ has $n$ non isomorphic
summands, see also \cite{Zhu2} and \cite{Wraa}.
The next useful result is due to A. Wr\aa{}lsen,\cite{Wraa}.
\begin{lemma}\label{wra}
Any $m$-cluster tilting object in $\mathcal{C}^m_{n}$ is induced
by a maximal $m$-rigid object in the fundamental domain of $\tau^{-1}[m]$
in $\mathcal{D}$.
\end{lemma}
\begin{prop}\label{tiltinmrep}
$T$ is an $m$-cluster tilting object in $\mathcal{C}^m_{n,p}$ if and only if
$$
\mathcal{X}^m_T=\mathcal{X}^m\cup\rho(\mathcal{X}^m)\cup\dots\cup\rho^{p-1}(\mathcal{X}^m)
$$
where $\mathcal{X}^m$ is an $m$-angulation of the region $\Pi_1$.
\end{prop}
\begin{proof}
First let $\mathcal{X}^m$ be an $m$-angulation
of a given region $\Pi_1$ in $\Pi^p.$ By Lemma \ref{mang} it corresponds
to a cluster tilting object $T_{\mathcal{C}^m_{n}}$ of $\mathcal{C}^m_{n}$. By Lemma \ref{wra} it
can be lifted to a maximal $m$-rigid object $T_{\mathcal{D}}$ of $\mathcal{D}$ contained in a 
fundamental domain of $\tau^{-1}[m]$.
Consider its $(\tau^{-1}[m])$-orbit $\{(\tau^{-1}[m])^i(T_{\mathcal{D}}) | i\in \Z \}$ 
in $\mathcal{D}$ which we denote again
by $T_{\mathcal{D}}$. Let $T_{\mathcal{C}^m_{n,p}}$
be the projection of $T_{\mathcal{D}}$ under the functor
$\pi^m_p:\mathcal{D}\rightarrow \mathcal{C}^m_{n,p}$.
Then  $T_{\mathcal{C}^m_{n,p}}$ is a direct sum of pairwise non isomorphic indecomposable objects,
which we denote by $T_1,\dots,T_t$ and we compute:

\begin{align*}
\mathrm{Ext}^i_{\mathcal{C}^m_{n,p}}(T_k,T_l)&=\mathrm{Hom}_{\mathcal{C}^m_{n,p}}(T_k,T_l[i])\\
&=\bigoplus_{j\in\Z}\mathrm{Hom}_{\mathcal{D}}(T_k,(\tau^{-p}[mp])^jT_l[i])\\
&=\bigoplus_{j\in\Z}\mathrm{Ext}^i_{\mathcal{D}}(T_k,(\tau^{-pj}[mpj])T_l)\\
&=0
\end{align*} for all $T_k, T_l$, $1\leq k,l\leq t$ and $1\leq i\leq m$.
In fact, the first and third equality hold by definition. The second follows
from the definition of morphisms in $\mathcal{C}^m_{n,p}$.
The last equality follows because if $T_{\mathcal{D}}$ is $m$-rigid then the same is true
for $(\tau^{-1}[m])(T_{\mathcal{D}})$, hence also for $(\tau^{-j}[mj])T_{\mathcal{D}}$.
And as $T_{\mathcal{D}}$ is maximal, it follows that in the fundamental domain of 
$(\tau^{-1}[m])$, $\mathrm{Ext}^i_{\mathcal{D}}(T_k,T_l)=0$ holds. Hence
$T_{\mathcal{C}^m_{n,p}}$  is $m$-rigid.

Now we show that $T_{\mathcal{C}^m_{n,p}}$ satisfies the second part of the definition
of $m$-cluster tilting objects in $\mathcal{C}^m_{n,p}$. For this we first
remark that the $\mathrm{Ext}$ spaces can be determined by similar computations as in 
Lemma \ref{exts}.
Secondly, by construction and Lemma \ref{mang} the objects $T_1,\dots,T_t$ of $\mathcal{C}^m_{n,p}$
correspond to a maximal set $\mathcal{X}^m_{T}$ of pairwise non crossing $m$-diagonals in $\Pi^p$. Furthermore, observe
that this set is given by the $\rho$-orbit of $\mathcal{X}^m$ in $\Pi^p$.
If  $T_{\mathcal{C}^m_{n,p}}$ was not $m$-cluster tilting, 
by the bijection given in Lemma \ref{transquivermdiag} 
there must be a region $\Pi_l$ in $\Pi^p$ where the set of $m$-diagonals is not maximal, 
but this contradicts the assumption that $\mathcal{X}^m$ was
an $m$-angulation. 
This proves one direction of the claim.

For the other direction, assume we have an $m$-angulation 
$\mathcal{X}^m_T$ of $\Pi^p$, and that $\mathcal{X}^m_T$
is not of the claimed form. 
Then $\mathcal{X}^m_T$ is such that there are regions $\Pi_k$ and $\Pi_l$ in $\Pi^p$
with $m$-angulations that are not obtained one from the other under $\rho$. 
Denote these two different $m$-angulations by $\mathcal{X}^m_{k}$ and $\mathcal{X}^m_{l}$ .
By Lemma \ref{mang} they correspond to $m$-cluster tilting objects in $\mathcal{C}^m_{n}$.

Consider the projection functor $\mu_p:\Pi^p\rightarrow\Pi$
sending a diagonal $D_X$ of  $\Pi^p$ to the corresponding diagonal $D_X$ in $\Pi$, and a morphism
between $D_X$ and $D_Y$ in $\Pi^p$ to the corresponding morphism
between $D_X$ and $D_Y$ in $\Pi$.
Then as  $\mu_p(\mathcal{X}^m_{k})\neq \mu_p(\mathcal{X}^m_{l})$
there are two $m$-diagonals, say $D_V$ and $D_W$ that cross in $\Pi$, 
because otherwise this would contradict the maximaliy of the $m$-angulation. Therefore, we conclude by
Lemma \ref{mang} that the corresponding objects $V$ and $W$
in $\mathcal{C}^m_{n}$ are such that $\mathrm{dim}(\mathrm{Ext}^i_{\mathcal{C}^m_{n}}(V,W))>0$ for an $1\leq i\leq m$.
Then, up to isomorphism, i.e. up to change the object $W$ with an other indecomposable object
in the same $\tau^{-1}[m]$-orbit in $\mathcal{D}$, 
it follows that also $\mathrm{dim}(\mathrm{Ext}^i_{\mathcal{C}^m_{n,p}}(V,W))>0$.
\noindent
This proves that the object in $\mathcal{C}^m_{n,p}$
corresponding to the diagonals $\mathcal{X}^m_T$ is not an $m$-cluster tilting
object.
\end{proof}
With this result we deduce immediately the following.
\begin{cor}
Any $m$-cluster tilting object in $\mathcal{C}^m_{n,p}$ contains $pn$ pairwise non isomorphic
summands.
\end{cor}

\section{Geometric model for $\mathcal{D}^b(\mathrm{mod}k A_{n+1})$}

Our strategy is to adapt the construction of
$\mathcal{C}^m_{n,p}$ given in the previous section. 
For this we will cut the large polygon $\Pi^p$ between two regions,
and let $p\rightarrow\infty$. In this way we obtain a figure with infinitely many sides
in which the idea is to no longer study $m$-diagonals
but their complements. More precisely, we consider the
diagonals which are not $2$-diagonals.  

Complements to $m$-diagonals in a given polygon $\Pi$
have been studied in \cite{Lisa1}. The result we need for the modelling of
$\mathcal{D}$ is Lemma \ref{Li}, it is based on the notion of the $m$-th power of a translation quiver introduced
in \cite{BM}. 

\subsection{ The $m$-th power of a translation quiver}
The {\em $m$-th power of $(\Gamma,\tau)$}, is a translation quiver $(\Gamma^m,\tau^m)$ which has
the same vertices as $\Gamma$, and whose arrows are given by paths of length $m$ in $\Gamma$:
$(x=x_0\rightarrow x_1\rightarrow\dots\rightarrow x_{m-1}\rightarrow x_m=y)$, 
such that whenever $\tau x_{i+1}$ is defined $\tau x_{i+1}\neq x_{i-1}$ for $i=1,\dots,m-1.$
Furthermore, the translation is given as $\tau^m:=\tau\circ\dots\circ\tau,$ $m$-times.
If $\Gamma$ is the AR-quiver of the cluster category of type $A_{(n+1)m-1}$, the $m$-th
power gives rise to $\Gamma^m_n$, the AR-quiver of the $m$-cluster category of type $A_n$, see \cite[Theorem 7.2]{BM}.

\begin{lemma}\label{Li} Let
$\Pi$ be a regular $(n+1)m+2$-gon and $\Gamma_{A_{m(n+1)-1}}$ the quiver of its diagonals. Then we have 
$$(\Gamma_{A_{m(n+1)-1}})^2=\Gamma_{n}^2\sqcup\Gamma_1\sqcup\Gamma_2,$$
where $\Gamma_1\cong\Gamma_2\cong AR(\mathcal{D}^b(\mathrm{mod} k A_{n+1})/[1])$, and
$[1]$ is the shift functor on $\mathcal{D}$.
\end{lemma}
\begin{proof}
This result was obtained from Corollary 4.13 and Theorem 4.16 in \cite{Lisa1}.  
\end{proof}
As a consequence of this result, one can say that
the vertices of
$\Gamma_{1}$ and $\Gamma_2$ are precisely the diagonals
whose endpoints are vertices of  $\Pi$  of the same parity.

\subsection{ The $\infty$-gon $\Pi^{\pm \infty}$}
In order to construct a polygon with infinitely many sides
we first introduce a polygon $\Pi^{\pm p}$ with $(2p+1)(2n+3)$ sides
and then let $p\rightarrow\infty$.

Let $\Pi_1$ be a $N:=(2(n+1)+2)$-gon and assume that its vertices
all lie on a line (c.f. Figure \ref{inftygon}).
Let $\varrho$ be the map shifting the figure $\Pi_1$,
$N$ steps to the right so that the left most vertex of $\varrho(\Pi_1)$
coincides with the right most vertex of $\Pi_1$.
Considering powers of $\varrho$ we obtain a figure with $2p+1$ regions homotopic to $\Pi_1$, each with vertices
numbered from $1$ up to $N$. Denote this figure
by $\Pi^{\pm p}$. 
The diagonals  in $\Pi^{\pm p}$
we will consider 
are inner arcs connecting even numbered vertices
in the same $\Pi_{k}$, $-p\leq k\leq p$. One could also have chosen the odd numbered vertices.
We call these diagonals {\em $2^c$-diagonals }.
Observe that the $2^c$-diagonals in any $\Pi_k$ give rise to a copy
of $\Gamma_1$ in Lemma \ref{Li}.
Letting $p\rightarrow \infty$ we obtain the $\infty$-gon $\Pi^{\pm \infty}$.

\subsection{The quiver $\Gamma^{\pm \infty}$}
We will define a translation quiver on the $2^c$-diagonals of $\Pi^{\pm \infty}$, for this 
we start associating a quiver to the $2^c$-diagonals of $\Pi^{\pm p}$. 
The arrows between $2^c$-diagonals
are defined in a similar way as in Section 6.
Let $(i,j,k)$ be a $2^c$-diagonal in a region $\Pi_k$ of $\Pi^{\pm p}$
with $j\neq N$, then  
$$
{
\xymatrix@-8mm{
 & & (1,j+2,k)\\
\ar[rru]\ar[rrd]  
(i,j,k)\\
 & & (i-2,j,k) 
}}
$$ 
where the first two entries are understood 
modulo $2$, and computations on the last entry have to be considered modulo $p$. 
Furthermore, for all $k\neq p$, we have arrows
$$
{
\xymatrix@-8mm{
 & &(2,i,k+1)\\
\ar[rru]\ar[rrd]  
(i,N,k)\\
 & & (i-2,N,k) 
}}
$$ 
where $(2,i,k+1)=\varrho(2,i,k)$. 
The condition on $k$ is needed because we want to
avoid that there is an arrow linking
the diagonals of the region $\Pi_p$ to those of the region $\Pi_{-p}$.
We will refer to the arrows $(i,N,k)\rightarrow(2,i,k+1)$, $k\neq p $ as {\em connecting arrows}.

Letting  $p\rightarrow \infty$ we obtain
infinitely many regions homotopic to $\Pi_1$, and we call the resulting figure $\Pi^{\pm \infty}$.
We denote the quiver of $2^c$-diagonals
of $\Pi^{\pm \infty}$ by $\Gamma^{\pm \infty}$.

\begin{figure}[htbp]
   \begin{center}
      \psfragscanon
\psfrag{Pi1}{$\Pi_1$}
\psfrag{Pi2}{$\Pi_2$}
\psfrag{Pi-2}{$\Pi_{-2}$}
\psfrag{varrho}{$\varrho$}
\psfrag{-varrho}{$\varrho^{-1}$}
\psfrag{dots}{$\dots$}
\psfrag{dots-}{$\dots$}
\centering
\includegraphics[width=0.7 \textwidth]{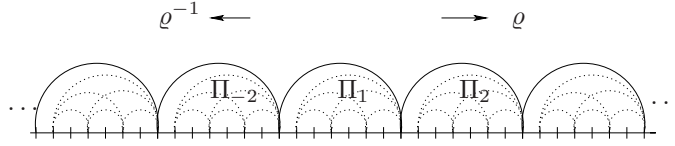}%
\qquad\qquad
\caption{ $\Pi^{\pm \infty}$ as model for $\mathcal{D}^b({\mathrm{mod}kA_3}).$}
\label{inftygon}
   \end{center}
\end{figure}

\subsection{Translation map for $\Gamma^{\infty}$}
We equip $\Gamma^{\pm\infty}$ with 
the translation $\tau_2$ from Definition
\ref{tau m} and let  $p$ tend to infinity.
More precisely, $\tau_2(i,j,k)=(i-2,j-2,k)$ for all $i,j$ different from $2$,
and $\tau_2(2,j,k)=\varrho^{-1}(N,j-2,k)=(j-2,N,k-1)$.
It can be proven as in
Lemma \ref{stabletransquiv} that this defines a stable translation quiver,
and taking the mesh category of it specifies morphisms for the
$2^c$-diagonals of $\Pi^{\pm \infty}$ arising from the mesh relations of 
$\Gamma^{\pm\infty}$.

\subsection{Geometric model of $\mathrm{mod} kA_{n+1}$ }

Observe that when restricting to the region $\Pi_1$ of $\Pi^{\pm \infty}$
we obtain the AR-quiver of the module category of type $A_{n+1}$ equipped with translation $\tau$.
In fact, write $\Gamma_{\Pi_1}$ for the full subquiver of
$\Gamma^{\pm \infty}$, with translation given by $\tau_2|_{\Gamma_{\Pi_1}}$, and assume
that the path algebra $k A_{n+1}$ is taken over an equioriented Dynkin quiver of type $A_{n+1}$.
\begin{lemma}\label{geomod} There is an isomorphism of stable
translation quivers
$$\Gamma_{\Pi_1}\cong AR(\mathrm{mod} kA_{n+1}).$$
\end{lemma}
\begin{proof}
By construction
$\Pi_1$ is homotopic to a regular $2(n+1)+2$-gon. Hence, the number of $2^c$-diagonals in $\Pi_1$
is $\frac{(n+2)(n+1)}{2}$. This number agrees with the isomorphism classes of
indecomposable objects in $\mathrm{mod} kA_{n+1}$.
Using Gabriel's correspondence between indecomposable objects in $\mathrm{mod} kA_{n+1}$
and positive roots of the associated Dynkin diagram we
associate positive roots to the $2^c$-diagonals of $\Pi_1$ and indecomposable modules to the positive roots.

Given a $2^c$-diagonal $(i,j,1)$ we 
let $\arrowvert i-j\arrowvert$ be its size. Then we
associate the $2^c$-diagonals of size $2$ to the simple roots $\alpha_i$, $1\leq i \leq n+1$ 
of the associated Dynkin diagram.
As the length of the segment augments, we increase the number of summands of the positive root.
Remark that the longest $2^c$-diagonal of $\Pi_1$ is $(2,2(n+1)+2)$, and it corresponds to $\alpha_1+\dots+\alpha_{n+1}$. 

Due to the shape of $\Gamma^{\textrm{even}}_{\Pi_1}$ and the well known shape of $AR(\mathrm{mod} kA_{n+1})$,
it is clear that the map we just defined
gives rise to an isomorphism of quivers, and it easy to check that
it preserves the translation so that it becomes an isomorphism of stable translation quivers.
\end{proof}

\subsection{Geometric model of the bounded derived category $\mathcal{D}$}

By taking the mesh category of $(\Gamma^{\pm\infty},\tau_2)$
we obtain a geometric model of $\mathcal{D}$.
\begin{thm}
There is an isomorphism of stable translation quivers
between the AR-quiver of $\mathcal{D}$ and
$(\Gamma^{\pm\infty},\tau_2)$.
\end{thm}

\begin{proof} By Lemma \ref{geomod} we know that
$$(\Gamma_{\Pi_1},\tau_2|_{\Gamma_{\Pi_1}})\cong (AR(\mathrm{mod} kA_{n+1}),\tau).$$ 
Furthermore, we observe
that  the translation $\varrho$  reproduces
copies of $\Gamma_{\Pi_1}$ in $\Gamma^{\pm\infty}$. 
Thus, it remains to study the connecting arrows. In
$\Gamma^{\pm\infty}$ these are
the arrows between  $2^c$-diagonals of the form $(i,N,k)$ and $(2, i-2 ,k+1) $ lying in the 
different regions $\Pi_k$
and $\Pi_{k+1}$ inside $\Pi^{\pm \infty}$. By construction
these arrows are exactly the connecting arrows between
different copies of the  AR-quiver 
of $\mathrm{mod}kA_{n+1}$ in $AR(\mathcal{D})$.
\end{proof}
\begin{cor}
There is a morphism of functors between $\varrho$ in the
category of $2^c$-diagonals of $\Pi^{\pm \infty}$
and the shift functor $[1]$ on $\mathcal{D}$.
\end{cor}
\begin{proof}
This follows directly from the arguments given in the proof of the previous result.
\end{proof}

\bibliographystyle{alpha}
\bibliography{referencesrep}

\newcommand{\etalchar}[1]{$^{#1}$}
\begin{thebibliography}{BMR{\etalchar{+}}06}

\bibitem[Ami07]{C.Amiot}
Claire Amiot.
\newblock On the structure of triangulated categories with finitely many
  indecomposables.
\newblock {\em Bull. Soc. Math. France}, 135(3):435--474, 2007.

\bibitem[BM07]{BM2}
Karin Baur and Robert~J. Marsh.
\newblock A geometric description of the {$m$}-cluster categories of type
  {$D_n$}.
\newblock {\em Int. Math. Res. Not. IMRN}, (4):Art. ID rnm011, 19, 2007.

\bibitem[BM08]{BM}
Karin Baur and Robert~J. Marsh.
\newblock A geometric description of {$m$}-cluster categories.
\newblock {\em Trans. Amer. Math. Soc.}, 360(11):5789--5803, 2008.

\bibitem[BMR{\etalchar{+}}06]{BMRRT}
Aslak~Bakke Buan, Robert Marsh, Markus Reineke, Idun Reiten, and Gordana
  Todorov.
\newblock Tilting theory and cluster combinatorics.
\newblock {\em Adv. Math.}, 204(2):572--618, 2006.

\bibitem[BT09]{BuTo}
Aslak~Bakke Buan and Hermund~Andr{\'e} Torkildsen.
\newblock The number of elements in the mutation class of a quiver of type
  {$D_n$}.
\newblock {\em Electron. J. Combin.}, 16(1):Research Paper 49, 23, 2009.

\bibitem[BZ10]{BZ}
T.~{Br{\"u}stle} and J.~{Zhang}.
\newblock {On the Cluster Category of a Marked Surface}.
\newblock {\em ArXiv e-prints}, (1005.2422), May 2010.

\bibitem[CCS06]{CCS}
P.~Caldero, F.~Chapoton, and R.~Schiffler.
\newblock Quivers with relations arising from clusters ({$A_n$} case).
\newblock {\em Trans. Amer. Math. Soc.}, 358(3):1347--1364, 2006.

\bibitem[FR05]{FoRe}
Sergey Fomin and Nathan Reading.
\newblock Generalized cluster complexes and {C}oxeter combinatorics.
\newblock {\em Int. Math. Res. Not.}, (44):2709--2757, 2005.

\bibitem[FZ02]{FZI}
Sergey Fomin and Andrei Zelevinsky.
\newblock Cluster algebras. {I}. {F}oundations.
\newblock {\em J. Amer. Math. Soc.}, 15(2):497--529 (electronic), 2002.

\bibitem[J{\o}r10]{Jorgensen}
Peter J{\o}rgensen.
\newblock Quotients of cluster categories.
\newblock {\em Proc. Roy. Soc. Edinburgh Sect. A}, 140(1):65--81, 2010.

\bibitem[Kel08]{Keller}
Bernhard Keller.
\newblock Calabi-{Y}au triangulated categories.
\newblock In {\em Trends in representation theory of algebras and related
  topics}, EMS Ser. Congr. Rep., pages 467--489. Eur. Math. Soc., Z\"urich,
  2008.

\bibitem[Kel10]{K2}
Bernhard Keller.
\newblock Cluster algebras, quiver representations and triangulated categories.
\newblock In {\em Triangulated categories}, volume 375 of {\em London Math.
  Soc. Lecture Note Ser.}, pages 76--160. Cambridge Univ. Press, Cambridge,
  2010.

\bibitem[{Lam}11]{Lisa1}
L.~{Lamberti}.
\newblock {A geometric interpretation of the triangulated structure of
  m-cluster categories}.
\newblock {\em ArXiv e-prints}, (1107.1041), July 2011.

\bibitem[PS00]{ps}
J{\'o}zef~H. Przytycki and Adam~S. Sikora.
\newblock Polygon dissections and {E}uler, {F}uss, {K}irkman, and {C}ayley
  numbers.
\newblock {\em J. Combin. Theory Ser. A}, 92(1):68--76, 2000.

\bibitem[Rie80]{Riedt}
C.~Riedtmann.
\newblock Algebren, {D}arstellungsk\"ocher, \"{U}berlagerungen und zur\"uck.
\newblock {\em Comment. Math. Helv.}, 55(2):199--224, 1980.

\bibitem[Sch08]{S}
Ralf Schiffler.
\newblock A geometric model for cluster categories of type {$D_n$}.
\newblock {\em J. Algebraic Combin.}, 27(1):1--21, 2008.

\bibitem[Tho07]{Thomas}
Hugh Thomas.
\newblock Defining an {$m$}-cluster category.
\newblock {\em J. Algebra}, 318(1):37--46, 2007.

\bibitem[Tor11]{Tork}
Hermund~Andr{\'e} Torkildsen.
\newblock Finite mutation classes of coloured quivers.
\newblock {\em Colloq. Math.}, 122(1):53--58, 2011.

\bibitem[Tza06]{Tzanaki}
Eleni Tzanaki.
\newblock Polygon dissections and some generalizations of cluster complexes.
\newblock {\em J. Combin. Theory Ser. A}, 113(6):1189--1198, 2006.

\bibitem[{van}10]{Va}
{A.-C.} {van Roosmalen}.
\newblock {Abelian hereditary fractionally Calabi-Yau categories}.
\newblock {\em ArXiv e-prints}, (1008.1245), August 2010.

\bibitem[Wra09]{Wraa}
Anette Wraalsen.
\newblock Rigid objects in higher cluster categories.
\newblock {\em J. Algebra}, 321(2):532--547, 2009.

\bibitem[Zhu08]{Zhu2}
Bin Zhu.
\newblock Generalized cluster complexes via quiver representations.
\newblock {\em J. Algebraic Combin.}, 27(1):35--54, 2008.

\bibitem[Zhu11]{Zhu}
Bin Zhu.
\newblock Cluster-tilted algebras and their intermediate coverings.
\newblock {\em Comm. Algebra}, 39(7):2437--2448, 2011.

\end{thebibliography}

\end{document}